\documentclass[12pt]{article}
\input{amssym}
\newtheorem{obs}{Observation}[section]
\newtheorem{thm}[obs]{Theorem}
\newtheorem{exm}[obs]{Example}

\newtheorem{lemma}[obs]{Lemma}

\newtheorem{deff}[obs]{Definition}
\newtheorem{ques}[obs]{Question}

\newtheorem{preproof}{{\bf Proof.}}

\newenvironment{proof}[1]{\begin{preproof}{\rm
               #1}\hfill{$\rule{2mm}{2mm}$}}{\end{preproof}}

\def\newpic#1{}
\date{}
\begin{document}

\title{
{\Large{\bf Graphs with constant adjacency dimension}}}
%

{\small
\author{Mohsen Jannesari
\\
[1mm]
{\small \it  Faculty of Basic Sciences}\\
{\small \it  University of Shahreza} \\
{\small \it 86149-56841, Shahreza, Iran}\\
{\small \it Email: mjannesari@shahreza.ac.ir}
}}

 \maketitle \baselineskip15truept

\begin{abstract}
For a set $W$ of vertices and a
vertex $v$ in a  graph $G$, the $k$-vector
$r_2(v|W)=(a_G(v,w_1),\ldots,a_G(v,w_k))$ is   the
{\it adjacency representation} of $v$ with respect to $W$, where $W=\{w_1,\ldots,w_k\}$ and $a_G(x,y)$
is the minimum of $2$ and the distance between the vertices $x$ and $y$. The set $W$ is
 an {\it adjacency resolving set} for $G$ if distinct vertices of $G$ have
distinct adjacency representations with respect to $W$. The minimum
cardinality of an adjacency resolving set for $G$ is its {\it adjacency dimension}. It is clear that the adjacency dimension of an $n$-vertex graph $G$ is between $1$ and $n-1$. The graphs with adjacency dimension $1$ and $n-1$ are known.
All  graphs with adjacency dimension $2$, and all $n$-vertex graphs with adjacency dimension $n-2$ are studied in this paper.
In terms of the diameter and order of $G$, a sharp upper bound is found for adjacency dimension of $G$. Also, a sharp lower bound  for
adjacency dimension of $G$ is obtained
in terms of order of $G$. Using these two bounds, all graphs with adjacency dimension 2, and all $n$-vertex graphs with adjacency dimension $n-2$
are characterized.
\end{abstract}

\section{Introduction}
Throughout this
paper, $G$ is a finite simple graph
with vertex set $V(G)$, edge set $E(G)$, and order $n(G)$. We use $\overline G$
for the complement of  $G$. The distance between two
vertices $u$ and $v$, denoted by $d_G(u,v)$, is the length of a
shortest path joining $u$ and $v$ in $G$, we write this  simply as
$d(u,v)$   when no confusion can arise. The diameter of $G$ is ${\rm diam}(G)=\max\{d(u,v)| u,v\in V(G)\}$.
 $N(v)$ is the set of all neighbors of vertex $v$.
The edge between adjacent vertices $u$ and $v$, is shown by $uv$. 
We use $P_n$ and $C_n$ to denote the isomorphism classes of $n$-vertex paths and cycles, respectively.  We use $v_1,\ldots,v_n$ to denote specific $n$-vertex paths  with vertices $v_1,\ldots,v_n$.
\par For  $W=\{w_1,\ldots,w_k\}\subseteq V(G)$ and a
vertex $v$ of $G$, the  $k$-vector
$$r(v|W)=(d(v,w_1),\ldots,d(v,w_k))$$
is   the {\it metric representation}  of $v$ with
respect to $W$. The set $W$ is  a {\it resolving set} for
$G$ if the vertices of $G$ have distinct metric representations, with
respect to $W$.
A resolving set $W$ for $G$ with
minimum cardinality is  a {\it metric basis} of $G$, and its
cardinality is the {\it metric dimension} of $G$, denoted by
$\dim(G)$.  The concepts of resolving
sets and metric
dimension of a graph
 were introduced independently by Slater~\cite{Slater1975}
and by Harary and Melter~\cite{Harary}. For more results
related to these concepts
see~\cite{baily,trees,K dimensional,bounds,sur1,landmarks}.
\par
Let $G$ and $H$ be two graphs with disjoint vertex sets. The {\it join} of  $G$ and $H$, denoted by $G\vee H$, is the
graph with vertex set  $V(G)\cup V(H)$ and edge set $E(G)\cup
E(H)\cup\{uv|\, u\in V(G),v\in V(H)\}$. Also, the {\it disjoint union} of  $G$ and $H$, denoted by $G\cup H$, is the
graph with vertex set  $V(G)\cup V(H)$ and edge set $E(G)\cup
E(H)$.
In~\cite{Ollerman} all graphs of order $n$ with metric dimension
 $n-2$ are characterized as follows.
\begin{thm} {\rm\cite{Ollerman}}\label{nn-2}
Let $G$ be a  connected graph of order $n\geq 4$. Then $\dim(G)=n-2$ if and
only if $G=K_{s,t},s,t\geq 1$, $G=K_s\vee\overline K_t,s\geq 1,
t\geq 2$, or $G=K_s\vee (K_t\cup K_1),s,t\geq 1$.
\end{thm}

During the study of  the metric dimension of lexicographic product of graphs, Jannesari and Omoomi~\cite{lexico} introduced the concept of {\it adjacency dimension} of graphs.
\begin{deff}\label{Adjacency dimension}\rm\cite{lexico}
Let $G$ be a graph, and $W=\{w_1,\ldots,w_k\}\subseteq V(G)$.
For each vertex $v\in V(G)$, the \emph{adjacency
representation} of $v$ with respect to $W$ is the $k$-vector
$${r_2}(v|W)=(a_G(v,w_1),\ldots,a_G(v,w_k)),$$ where
$$a_G(v,w_i)=\left\{
\begin{array}{ll}
0 &  {\rm if}~v=w_i, \\
1 &  {\rm if}~v~{\rm is~adjacent~to~} w_i,\\
2 &   {\rm otherwise}.
\end{array}\right.$$
The set $W$ is an \emph{adjacency  resolving set} for $G$
if the vectors ${r_2}(v|W)$ for $v\in V(G)$ are distinct.
The minimum cardinality of an adjacency resolving set
is the \emph{adjacency dimension} of $G$, denoted by
$\dim_2(G)$. An adjacency resolving set of cardinality
$\dim_2(G)$ is  an \emph{adjacency basis} of $G$.
\end{deff}
We say that a set $W$ {\it  (adjacency) resolves} a set $T$ of vertices
in $G$, if the  adjacency representations of vertices in $T$ with respect to W are distinct.
  To determine whether a given set $W$ is  an adjacency resolving
set for $G$, it is sufficient to look at the  adjacency  representations of
vertices in $V(G)\backslash W$, because $w\in W$ is the unique
vertex of $G$ for which($a_{_G}(w,w)=0$) $d(w,w)=0$.
\par

 After the introducing of adjacency dimension, researchers interested in studying  this parameter, use it for finding metric dimension of some families of graphs and defined some related parameter to adjacency dimension.
Fernau and Rodriguez~\cite{fernando,fernau} use adjacency dimension to show that the metric dimension of the corona product of a graph of order $n$ and some nontrivial graph $H$ is equal to $n$ times the adjacency dimension of $H$. Using this relationship, they showed that the problem of computing the adjacency dimension is $NP$-hard. They also define a new related parameter, {\it local adjacency dimension} and use it to  show that the local metric dimension of the corona product of a graph of order $n$ and some nontrivial graph $H$ is equal to $n$ times the local adjacency dimension of $H$.
Estrada et al.~\cite{on adjacency dimension} introduced the concept of {\it $k$-adjacency dimension} and obtained some bounds and closed formulas for some families of graphs. In particular they obtained a closed formula for the $k$-adjacency dimension of join graphs.
\par By previous works, it is clear that each result about adjacency dimension is important for the study of metric dimension of lexicographic product graphs and corona product graphs.
In this paper, we find  sharp upper and lower bounds for adjacency dimension in terms of diameter and order of a graph.
  These bounds give us new useful bounds for metric dimension of lexicographic product graphs and corona product graphs.
\par It is clear that for each graph $G$, $1\leq\dim_2(G)\leq n(G)-1$. All graphs with adjacency dimension $1$ and all graphs with adjacency dimension $n(G)-1$ are characterized as the following lemma.
\begin{lemma}\label{chr=1,n-1}\rm\cite{lexico} Let $G$ be a graph of order $n$.
\begin{itemize}
\item $\dim_2(G)=n-1$ if and only if $G=K_n$ or $G=\overline{K}_n$.
\item $\dim_2(G)=1$ if and only if $G\in\{P_1,P_2,P_3,\overline{P}_2,\overline{P}_3\}$.
\end{itemize}
\end{lemma}
In Section~\ref{n-2}, we find a new sharp upper bound for adjacency dimension in terms of the diameter  and order of a graph. Using this bound and characterization of graphs $G$ with metric dimension $n(G)-2$, a characterization of graphs $G$ with adjacency dimension $n(G)-2$ is obtained. In Section~\ref{b2=2}, a new sharp lower bound for adjacency dimension in terms of the order of a graph is presented. All graphs that attain this bound are characterized. This bound and some known results lead us to a characterization of graphs  with adjacency dimension $2$.
\par
The next
results about adjacency dimension  of graphs is needed in the following.
\begin{lemma}\label{adjacencyresults}\rm\cite{lexico} Let $G$ be a graph of order $n$.
\begin{itemize}

\item If  ${ diam}(G)=2$, then $\dim_2(G)=\dim(G)$.

\item If  $G$ is connected, then $\dim_2(G)\geq \dim(G)$.

\item $\dim_2(G)=\dim_2(\overline{G})$.
\item If $n\geq 4$, then $\dim_2(C_n)=\dim_2(P_n)=\lfloor{{2n+2}\over 5}\rfloor$.
\end{itemize}
\end{lemma}
Two distinct vertices $u$ and $v$ are  {\it twins}
if $N(v)\backslash\{u\}=N(u)\backslash\{v\}$. 
It is easy to see that, if $u,v$ are twins in $G$ then for each $x\in V(G)\backslash\{u,v\}$, $d(x,u)=d(x,v)$ and therefore
$a_G(x,u)=a_G(x,v)$. Thus, we have the following lemma.
\begin{lemma}\label{twinadjacency}
If $u,v$ are twin vertices in a graph $G$, then every adjacency resolving set for $G$ contains at least one of the vertices $u$ and $v$.
\end{lemma}
\section{Graphs of order $n$ and adjacency dimension $n-2$}\label{n-2}

This section is aimed to characterize all $n$-vertex graphs with adjacency dimension $n-2$. To gain this goal, we first find a sharp
upper bound for adjacency dimension of graphs in terms of its order and diameter.
 \begin{lemma}\label{beta2<m-d-1+2D+4/5}
Let $G$ be a connected graph of order $n$ and diameter $D$. Then $$\dim_2(G)\leq n-D-1+\lfloor{{2D+4}\over 5}\rfloor.$$
\end{lemma}
\begin{proof}{
If $D=1$, then $G=K_n$ and $\dim_2(G)=n-1= n-D-1+\lfloor{{2D+4}\over 5}\rfloor$. Now we consider $D\geq2$.
Let $u,v\in V(G)$ be two vertices with $d(u,v)=D$ and $P_{D+1}=u,u_2,u_3,\ldots,u_{_D},v$ be a shortest path between $u$ and $v$.
If $D=2$, then $\dim_2(P_3)=1=\lfloor{{2D+4}\over 5}\rfloor$. For $D\geq3$, by
 Lemma~\ref{adjacencyresults}, $\dim_2(P_{D+1})=\lfloor{{2D+4}\over 5}\rfloor=t$.
 Let $B=\{v_1,v_2,\ldots,v_t\}\subseteq\{u,u_2,u_3,\ldots,u_D,v\}$
 be an adjacency basis of $P_{D+1}$. Consider the set $W=(V(G)-V(P_{D+1}))\cup B$. If there exist vertices $x,y\in V(G)-W\subseteq V(P_{D+1})$ with $r_2(x|W)=r_2(y|W)$, then  $r_2(x|B)=r_2(y|B)$ and this is a contradiction, because $B$ is an adjacency basis of $P_{D+1}$.
 Thus, $W$ is an adjacency resolving set for $G$ with cardinality $n-D-1+\lfloor{{2D+4}\over 5}\rfloor$.
}\end{proof}

It is clear that upper bound in Theorem~\ref{beta2<m-d-1+2D+4/5} is  tight for  $G=P_n$
and $G=K_n$.
In the next theorem, we construct an infinite family of graphs  of diameter $D$, order $n\geq D+1$ and adjacency dimension
$n-D-1+\lfloor{{2D+4}\over 5}\rfloor$. Therefore upper bound in
 Lemma~\ref{beta2<m-d-1+2D+4/5}  is sharp.
\begin{thm}
Let $k$ be a positive integer and $D\in\{5k,5k+2\}$. Then for each integer $n\geq D+1$ there exists a graph $G$
with $n$ vertices and diameter $D$, such that $\dim_2(G)=n-D-1+\lfloor{{2D+4}\over 5}\rfloor.$
\end{thm}
\begin{proof}{
Let $G$ be a graph with $V(G)=\{v_0,v_1,v_2,\ldots,v_{_D}\}\cup\{u_1,u_2,\ldots,u_{_{n-D-1}}\}$ and
$E(G)=\{v_iv_{i+1}|0\leq i\leq D-1\}\cup\{v_iu_j|0\leq i\leq2, 1\leq j\leq n-D-1\}\cup\{u_iu_j|1\leq i,j\leq n-D-1\}$, see Figure~\ref{figure}.
Clearly $G$ has $n$ vertices and $diam (G)=D$.
We prove that $\dim_2(G)=n-D-1+\lfloor{{2D+4}\over 5}\rfloor$.
\par
If $n=D+1$, then there is no any vertex $u_i$ in $G$ and $G=P_{_{D+1}}$. Thus, by Lemma~\ref{adjacencyresults}
$\dim_2(G)=\lfloor{{2(D+1)+2}\over 5}\rfloor=n-D-1+\lfloor{{2D+4}\over 5}\rfloor$.
\par
Now let $n\geq D+2$ and $B$ be an adjacency basis of $G$. In this case, the set $U=\{v_1,u_1,u_2,\ldots,u_{_{n-D-1}}\}$ is
a set of twin vertices. So, by Lemma~\ref{twinadjacency}, at most one of the vertices of $U$ can be not in $B$, say $U-B\subseteq\{u\}$.
Note that, $r_2(u|U-\{u\})=r_2(v_0|U-\{u\})=r_2(v_2|U-\{u\})=(1,1,\ldots,1)$ and $r_2(v_i|U-\{u\})=(2,2,\ldots,2)$, for $i\geq 3$.
Since none of vertices $v_3,v_4,\ldots,v_{_D}$ can  adjacency resolve vertices $v_0$ and $v_1$, to adjacency resolve $\{u,v_0,v_2\}$, we need at least one vertex from this set, say $y$. But, $r_2(v_i|(U\cup\{y\})-\{u\})=(2,2,\ldots,2)$, for $i\geq 4$.
Since $v_i$'s, $4\leq i\leq D$, form a path of order $D-3$, by Lemma~\ref{adjacencyresults} to adjacency resolve these vertices we need
$\lfloor{{2(D-3)+2}\over 5}\rfloor$ vertices from this set. Therefore
$$\dim_2(G)\geq |(U\cup\{y\})-\{u\}|+\lfloor{{2(D-3)+2}\over 5}\rfloor=n-D+\lfloor{{2(D-3)+2}\over 5}\rfloor=n-D+\lfloor{{2D-4}\over 5}\rfloor.$$
On the other hand, if $B_1$ is a basis of $P_{_{D-3}}=v_4,v_5,\ldots,v_{_D}$, then the set $\{v_2,u_1,u_2,\ldots,u_{_{n-D-1}}\}\cup B_1$ is an adjacency resolving set for $G$ of size $n-D+\lfloor{{2D-4}\over 5}\rfloor$. Therefore, $\dim_2(G)=n-D+\lfloor{{2D-4}\over 5}\rfloor$.
\par
In case $D=5k$, we have $$\lfloor{{2D-4}\over 5}\rfloor=\lfloor{{10k-4}\over 5}\rfloor=2k-1=\lfloor{{10k+4}\over 5}\rfloor-1=\lfloor{{2D+4}\over 5}\rfloor-1.$$
And in case $D=5k+2$,
$$\lfloor{{2D-4}\over 5}\rfloor=\lfloor{{10k+4-4}\over 5}\rfloor=2k=\lfloor{{10k+4+4}\over 5}\rfloor-1=\lfloor{{2D+4}\over 5}\rfloor-1.$$
Therefore in these two cases, $\dim_2(G)=n-D-1+\lfloor{{2D+4}\over 5}\rfloor$.
}\end{proof}
 For arbitrary $D$ and each $n\geq D+1$  the following question is propounded.
\begin{ques}
Is there a graph $G$ of diameter $D$ with $\dim_2(G)=n-D-1+\lfloor{{2D+4}\over 5}\rfloor$, for each $D$ and  $n\geq D+1$.
\end{ques}

 All graphs of order $n$ with metric dimension
 $n-2$ are characterized in Theorem~\ref{nn-2}.
Through the next theorem all graphs of order $n$ and adjacency dimension $n-2$ are characterised.
\begin{thm}\label{b2=n-2}
Let $G$ be a graph of order $n$. Then $\dim_2(G)=n-2$ if and only if $G$ or $\overline G$ is one of the graphs $P_4$,
 $K_{s,t}~(s,t\geq 1), K_s\vee\overline K_t~(s\geq 1,
t\geq 2)$, or $K_s\vee (K_t\cup K_1)~ (s,t\geq 1)$.
\end{thm}
\begin{proof}{
 If $G$ or $\overline G$ is one of the graphs $P_4$,
 $K_{s,t}~(s,t\geq 1), K_s\vee\overline K_t~(s\geq 1,
t\geq 2)$, or $K_s\vee (K_t\cup K_1)~ (s,t\geq 1)$, then it is clear that
$\dim_2(G)=n-2$.
\par
Conversely, we first prove for connected graphs. Let $G$ be a connected graph of order $n$ and $\dim_2(G)=n-2$. If $diam(G)\geq 4$, then by
Lemma~\ref{beta2<m-d-1+2D+4/5}, $$n-2\leq n-4-1+\lfloor{{8+4}\over 5}\rfloor=n-3.$$
This contradiction implies $diam(G)\leq3$. If $diam(G)\leq2$, then by Lemma~\ref{adjacencyresults}
$\dim_2(G)=\dim(G)$, thus by Theorem~\ref{n-2}
$G$  is one of the graphs,
 $K_{s,t}~(s,t\geq 1), K_s\vee\overline K_t~(s\geq 1,
t\geq 2)$, or $K_s\vee (K_t\cup K_1)~ (s,t\geq 1)$.
\par
Now let $diam(G)=3$ and $x,y\in V(G)$ such that $d(x,y)=3$. Suppose that $N_i(x)=\{t\in V(G)|d(x,t)=i\}$, $0\leq i\leq3$. If there exist non-adjacent vertices $b\in N_2(x)$ and $c\in N_3(x)$, then $c$ has a neighbour  $b'\neq b$ in $N_2(x)$ and
$$r_2(b|\{x,c\})=(2,2),~~~~ r_2(b'|\{x,c\})=(2,1),~~~~ r_2(e|\{x,c\})=(1,2),$$
where $e$ is an arbitrary vertex in $N_1(x)$.
Since these three adjacency representations are distinct,
 $V(G)-\{b,b',e\}$ is an adjacency resolving set for $G$. This contradiction implies that
all vertices of $N_3(x)$ are adjacent to all vertices in $N_2(x)$.
If there exist non-adjacent vertices $b\in N_1(x)$ and $c\in N_2(x)$, then $c$ has a neighbour  $b'\neq b$ in $N_1(x)$ and
$$r_2(b|\{x,c\})=(1,2),~~~~ r_2(b'|\{x,c\})=(1,1),~~~~ r_2(e|\{x,c\})=(2,1),$$
where $e$ is an arbitrary vertex in $N_3(x)$. Hence $V(G)-\{b,b',e\}$ is an adjacency resolving set for $G$. This contradiction implies that
all vertices of $N_2(x)$ are adjacent to all vertices in $N_1(x)$. If $u,v\in N_1(x)$ are two distinct vertices, then
$$r_2(y|\{x,u\})=(2,2),~~~~ r_2(z|\{x,u\})=(2,1),~~~~ r_2(v|\{x,u\})=(1,i),$$
where $z$ is an arbitrary vertex in $N_2(x)$ and $i\in\{1,2\}$. Hence $V(G)-\{y,z,v\}$ is an adjacency resolving set for $G$. This contradiction implies that $|N_1(x)|=1$, say $N_1(x)=\{w\}$. If $u,v\in N_2(x)$ are two distinct vertices, then
$$r_2(x|\{y,u\})=(2,2),~~~~ r_2(w|\{y,u\})=(2,1),~~~~ r_2(v|\{y,u\})=(1,i),$$
where  $i\in\{1,2\}$. Hence $V(G)-\{x,w,v\}$ is an adjacency resolving set for $G$. This contradiction implies that $|N_2(x)|=1$, say $N_2(x)=\{z\}$.
If $u,v\in N_3(x)$ are two distinct vertices, then
$$r_2(x|\{w,u\})=(1,2),~~~~ r_2(z|\{w,u\})=(1,1),~~~~ r_2(v|\{w,u\})=(2,i),$$
where $i\in\{1,2\}$. Hence $V(G)-\{x,z,v\}$ is an adjacency resolving set for $G$. This contradiction implies that $|N_3(x)|=1$. Therefore $G=P_4$.
\par
Since the complement of a disconnected graph is connected,
$G$ or $\overline G$ is one of the graphs $P_4$,
 $K_{s,t}~(s,t\geq 1), K_s\vee\overline K_t~(s\geq 1,
t\geq 2)$, or $K_s\vee (K_t\cup K_1)~ (s,t\geq 1)$.
}\end{proof}
\section{Graphs with adjacency dimension $2$}\label{b2=2}
In  this Section, a sharp upper bound for order of graphs with 
adjacency dimension $k$ is presented. This leads us to a lower bound for adjacency dimension of graphs
in terms of order of a graph. 
In fact if we consider $f(n)$ as the least positive integer $k$ such that $k+2^k\geq n(G)$,
then for each graph $G$ of order $n$, $\dim_2(G)\leq f(n)$.  All graphs that attain this bound are characterized. This bound and some known results lead us to a characterization of graphs  with adjacency dimension $2$.
\begin{lemma}\label{n<b2+2^b2}
Let $G$ be a graph of order $n$ and $\dim_2(G)=k$. Then $n\leq k+2^k$.
\end{lemma}
\begin{proof}{
Let $B$ be an adjacency basis of $G$. For each vertex $v\in V(G)$, $r_2(v|B)$ is a $k$-vector with entries $0,1,2$.
The members of $B$ are all vertices that their representation have entry $0$. The representation of  other vertices are constructed by $1$ and $2$.
Therefore, $G\setminus B$ has at most $2^k$ vertices. That is $n\leq k+2^k$.
}\end{proof}
In the next example we construct a family of graphs with $6=2+2^2$ vertices and adjacency dimension $2$.
\begin{exm}\label{k=2}
Let $n(G)=6$,  we introduce a construction for graphs with  adjacency dimension $2$. Let $H$ and $K$ be two arbitrary graphs
with $V(H)=\{a,b\}$ and $V(K)=\{c,d,e,f\}$. Consider graph $G$ with $V(G)=\{a,b,c,d,e,f\}$ and $E(G)=E(H)\cup E(K)\cup\{ac,ad,bd,be\}$.
By Lemma~\ref{n<b2+2^b2}, $\dim_2(G)\geq2$. On the other hand,
$$r_2(c|\{a,b\})=(1,2),~~~r_2(d|\{a,b\})=(1,1),~~~r_2(e|\{a,b\})=(2,1),~~~r_2(f|\{a,b\})=(2,2).$$
Hence, $\{a,b\}$ is an adjacency resolving set for $G$ of size $2$. Therefore $\dim_2(G)=2$.
\end{exm}
By extending the construction in Example~\ref{k=2},  we construct an infinite family of graphs with adjacency dimension $k$ and order $k+2^k$ in the next theorem. In fact in this theorem for each positive integer $k$, we find all graphs with adjacency dimension $k$ and order $k+2^k$.
\begin{thm}\label{thm n=b2+2^b2}
For each positive integer $k$, the family $\Omega_k$ of all graphs with adjacency dimension $k$ and order $k+2^k$
has $2^{{k\choose 2}+{2^k\choose 2}}$ members.
\end{thm}
\begin{proof}{
Let $k$ be a positive integer and $G$ be a graph with vertex set $V(G)=\{v_1,v_2,\ldots ,v_k\}\cup \{u_1,u_2,\ldots, u_{_{2^k}}\}$, where $u_i$'s are
$k$-vectors with entries $1$ and $2$. A vertex $v_i$ is adjacent to $u_j$ if the $i$th entry of $u_j$ is $1$. The adjacency of two members of $V=\{v_1,v_2,\ldots ,v_k\}$ is arbitrary. Also the adjacency of two members of $U=\{u_1,u_2,\ldots, u_{_{2^k}}\}$ is arbitrary. We define the family $\Omega_k$ all these graphs. Since the adjacency of two members of $V$ and two members of $U$ are arbitrary, $|\Omega_k|=2^{{k\choose 2}+{2^k\choose 2}}$. Now we prove that the family $\Omega_k$ consists of all graphs with adjacency dimension $k$ and order $k+2^k$.
\par Let $G$ be a graph in $\Omega_k$, by Lemma~\ref{n<b2+2^b2}, $\dim_2(G)\geq k$. On the other hand, for each $u_i$, $1\leq i\leq 2^k$, the adjacency representation of $u_i$ with respect to $V$ is its corresponding $k$-vector, that is the adjacency representation of all vertices  of $U$ with respect to $V$ are deferent. Also for each $i$, $1\leq i\leq k$, $v_i$ is the unique vertex of $G$ with $0$ in $i$th entry in $r_2(v_i|V)$, thus $V$ is an adjacency resolving set for $G$ of size $k$. Therefore the adjacency dimension of all members of $\Omega_k$ is $k$.
\par Now we need to prove that each graph with adjacency dimension $k$ and order $k+2^k$ belongs to $\Omega_k$. Let $G$ be a graph with this properties, $V$ be its adjacency basis and $U=V(G)\setminus V$. Hence $|V|=k$ and there are $2^k$ distinct adjacency representations with respect to $V$ for vertices in $U$. Since $|U|=2^k$, the set of adjacency representations of vertices in $U$  with respect to $V$
is the set of all $k$-vectors with entries $1$ and $2$.  That is, $U$ is correspond to the set of all $k$-vectors with entries $1$ and $2$. Also a vertex $v_i\in V$ is adjacent to a vertex $u_j\in U$ if and only if the $i$th entry of its adjacency representation is $1$. Therefore $G\in \Omega_k$.
}\end{proof}
In the remaining of this section we find all graphs with adjacency dimension $2$.
By Lemma~\ref{n<b2+2^b2}, if $\dim_2(G)=2$, then $n(G)\leq6$.
Thus to study graphs with adjacency dimension $2$, it is sufficient to consider all graphs with at most $6$ vertices. In the next, all graphs $G$ with order at most $6$ are studied and in each case all graphs with adjacency dimension $2$ are detected.
\par Case 1. $n(G)\leq2$: in this case, all graphs has adjacency dimension $1$.
\par
Case 2. $n(G)=3$, in this case, $n-1=2$ and by Lemma~\ref{chr=1,n-1}, $\dim_2(G)=2$ if and only if $G=K_3$ or $G=\overline K_3$.
\par
Case 3.  $n(G)=4$, in this case by Lemma~\ref{n<b2+2^b2}, $\dim_2(G)\geq2$. Clearly
$\dim_2(G)\leq n(G)-1=3$. But $\dim_2(G)=3$ if and only if $G=K_4$ or $G=\overline K_4$. Therefore $\dim_2(G)=2$ if and only if $G$ is
not $K_4$ or $\overline K_4$.
\par
Case 4. $n(G)=5$, Lemma~\ref{n<b2+2^b2} implies that $\dim_2(G)\geq2$, clearly
$\dim_2(G)\leq n(G)-1=4$.
 Also, by Lemma~\ref{chr=1,n-1}  $\dim_2(G)=4$ if and only if $G=K_5$ or $G=\overline K_5$.
Moreover, Theorem~\ref{b2=n-2} implies $\dim_2(G)=3$ if and only if $G$ or $\overline G$ is one of the graphs,
$K_{1,4},K_{2,3},K_3\vee\overline{K}_2,K_2\vee\overline{K}_3,K_1\vee(K_3\cup K_1)$ or $K_2\vee(K_2\cup K_1)$.
Therefore, $\dim_2(G)=2$ if and only if  $G$ or $\overline{G}$ is not any of the following graphs.
$$K_5,K_{1,4},K_{2,3},K_3\vee\overline{K}_2,K_2\vee\overline{K}_3,K_1\vee(K_3\cup K_1),K_2\vee(K_2\cup K_1).$$
\par
Case 5. $n(G)=6$, in this case the family of all graphs is $\Omega_2$. In fact these graphs are described in Example~\ref{k=2}.
\par Therefore we found all graphs with adjacency dimension $2$.

\end{document}